\RequirePackage[warning,log]{snapshot}
\documentclass{amsart}
\usepackage{graphicx}

\newtheorem{theorem}{Theorem}[section]
\newtheorem{lemma}[theorem]{Lemma}

\theoremstyle{definition}

\newtheorem{example}[theorem]{Example}

\theoremstyle{remark}
\newtheorem{remark}[theorem]{Remark}

\numberwithin{equation}{section}

\begin{document}

\title{Chebyshev polynomials and higher order Lucas Lehmer algorithm  }
\author{Kok Seng Chua}

\email{chuakkss52@outlook.com}

\




\begin{abstract}
We extend the necessity part of Lucas Lehmer iteration for testing Mersenne prime to all base and uniformly for both generalized Mersenne and Wagstaff numbers(the later correspond to negative base). The role of the quadratic iteration
$x \rightarrow x^2-2$ is extended by Chebyshev polynomial $T_n(x)$ with an implied iteration algorithm because of  the compositional identity $T_n(T_m(x))=T_{nm}(x)$. This results from a Chebyshev polynomial primality test based essentially on the Lucas pair $(\omega_a,\overline{\omega}_a)$, $\omega_a=a+\sqrt{a^2-1}$, where $a \neq 0 \pm 1$.  It gives a uniform way to detect primality of all integers of the form $\Phi_{p}(q,r):=\frac{q^p-r^p}{q-r}$ for $q \neq 0,\pm 1$ and $gcd(q,r)=1$ which implies for example $T_{q^p-r^p}(a)=T_{q-r}(a)$ mod $\Phi_p(q,r)$ for any $a \neq 0, \pm 1$. The Chebyshev test using $T_n(x)$ is a natural extension of the usual Fermat test using $t_n(x):=x^n$ which satisifies the simplest instance of  the composition law $t_n(t_m)=t_{nm}$.

To test primality of $Q$, the method essentially do a Fermat little test in the ring $\mathbb{Z}/(Q \mathbb{Z})[\sqrt{a^2-1}]$ with unit base $\omega_a$. The advantage is that when we change the base unit $\omega_a$, it also changes the ring which gives more possibility. The integers in a quadratic ring has two components and it is more convenient to work with rational integers by taking trace, which is the  Chebyshev polynomial. 

We further observed that there is a natural generalization of  Mersenne prime search  to multi-prime parameters search given by
general homogenized cyclotomic with odd square free index $\Phi_n(r,s)$. This has explicit product form. As an example, the cyclotomic number
$$\Phi_{2021}(4,13)=\frac{(4-13)(4^{2021}-13^{2021})}{(4^{43}-13^{43})(4^{47}-13^{47})},$$
is a $2152$ digit prime.

We also note that Chebyshev polynomials $T_n,U_n$  have  a twisted bosonic version $S_N,V_n$ and  they can all be derived as odd and even part of a bionomial $(1+x)^n$ following a n MO solution on real-rootedness of $s$ Eulerian polynomial.
\end{abstract}

\maketitle
\section{Main results and proof} The Chebyshev polynomial of the first kind, can be  explicitly defined \cite{WC}  for $|x| \ge 1$, by
$$T_n(x)=\frac{(x+\sqrt{x^2-1})^n+(x-\sqrt{x^2-1})^n}{2}=\frac{\omega_x^n+\overline{\omega}_x^n}{2},$$
where $\omega_x:=x+\sqrt{x^2-1}$, has a natural extension to negative value of $n$, with $T_{-n}(x)=T_n(x)$ since  $\omega \overline{\omega}=1$. In fact we can extend $T_n(x)=\cosh(n \log(\omega_x))$ to all real or even complex value of $n$ and also $x$ (with the recursion
$T_{n+1}(x)=2xT_n(x)-T_{n-1}(x)$ still holds). This also says that for integral $a \neq 0, \pm 1$, $T_n(a)$ is  (half) the trace of the $n$th power of the unit $\omega_a$ in $\mathbb{Q}(\sqrt{a^2-1})$, and indeed for all such $a$, $(\omega_a,\overline{\omega}_a)$ is a Lucas pair in the sense of \cite{BHV}. This seems to suggest that apart from its role in numerical analysis, the Chebyshev polynomial $T_n(x)$ may also have some interesting arithmetical properties.

If $n$ is a positive integer, we have clearly $\omega_a^n=P_n(a)+Q_n(a)\sqrt{a^2-1}$ for some $P_n(x), Q_n(x) \in \mathbb{Z}[x]$, but they are just Chebyshev polynomials
$$\omega_a^n:=(a+\sqrt{a^2-1})^n=T_n(a)+U_{n-1}(a)\sqrt{a^2-1}$$ which follows from
$\omega_a^{n+1}=\omega_a^n\omega_a$ and induction, and this is another way to think of  $T_n(x),U_n(x)$ which also derives the  recursion which one can never remember. It also suggest a twisted version of Chebyshev if we replace $\omega_a$ with the negative unit $a+\sqrt{a^2+1}$. Since $\omega_a$ is a unit, so is $\omega_a^n$ and hence we have the Pell's equation $T_n(a)^2-U_{n-1}(a)^2(a^2-1)=1$. It follows that we have (also deducible from the elementary identity $x=T_1(x)=\cosh (\log (\omega_x))$),
$$\omega_a^n=T_n(a)+\sqrt{T_n(a)^2-1}=\omega_{T_n(a)},$$
and this implies the compositional multiplicativity $T_m(T_n(a))=T_{mn}(a)$.

This seems to be a remarkable property which is special only to Chebyshev and not the other orthogonal polynomials. It implies that one can compute $T_{q^n}(a)$ efficiently as $T_q^n(a)$, where the power on the right means compositional iteration of the smaller polynomial $T_q(a)$.
In particular if we let $s_n=2T_{2^n}(2)$, then using $T_2(x)=2x-1$, we get
\begin{equation} s_0=4,\;\;\;s_{n+1}=s_n^2-2,
\end{equation}
which is exactly Lucas-Lehmer iteration for testing Mersenne prime. So we have for odd prime $p$, $M_p=2^p-1$ is prime if and only if $M_p$ divides $2T_{2^{p-2}}(2)$ ($= s_{p-2}$).

It seems that there is no reason for the base $2$ to be special for Chebyshev and it is natural to wonder if there is a natural test for primality of $q^p-1$ or rather $\Phi_p(q):=(q^p-1)/(q-1)$ (dividing out the obvious factor) for $q>2$ using perhaps $T_{q^{n-2}}(q)$ mod $\Phi_p(q)$, using the implied fast computation.

A little experimentation testing initial primes $p$ upto $300$, led us to  visually discover that there is always a dip (similar to the dip in light intensity in a well known method for detecting exo-planets) in the number of digits of the residues $T_{q^n}(a)$ mod $\Phi_p(q)$ for $a \neq 0, \pm 1$, exactly for primes $p$ where $\Phi_p(q)$ is prime, but there appears to be two distinct residues. Checking some examples show a dependence on the quadratic character
$\left( \frac{a^2-1}{\Phi_p(q)} \right)$ (which always equals $-1$ in the Lucas-Lehmer case where $q=a=2$ since $3$ is a qudratic non residue mod any Mersenne prime). This then allows us to guess the correct statement (1.2) numerically in the theorem below. An unexpected addition, which we discovered accidentally is that the result still holds when $q \neq 0, \pm 1$ is negative, in which case it corresponds to primality
of generalized Wagstaff primes $\Phi_p(-q)=(q^p+1)/(q+1)$ with the same algorithm as long as $T_{-n}(x)$ is implemented as $T_n(x)$.

It is now easy to deduce the proof of $(1.2)$ from the standard proof of the necessity part of Lucas - Lehmer replacing $\omega_2=2+\sqrt{3}$ with $\omega_a=a+\sqrt{a^2-1}$ once we know we need to keep track of the quadratic character (Lemma 1.6).
In this instance, knowing the exact statement of the result (which we found numerically) lead us to the generalized proof. Theorem (1.2) is however weaker than the necessity part of Lucas -Lehmer in the case $q=a=2$, which however can be deduced from Lemma1.6. We keep the statement of (1.2) for its simplicity.

\begin{theorem} Let $q,a$ be integers both not $0,\pm 1$ and $p$ be an odd prime and let $\epsilon=\left( \frac{a^2-1}{\Phi_p(q)} \right)$, if $\Phi_p(q)$ is a prime not dividing $a^2-1$, then
  \begin{equation}
  T_{q^p}(a)=T_{q+\epsilon-1}(a),\;\; U_{q^p-\epsilon (q-1)-2}(a)=0 \;\;\; mod \;\; \Phi_p(q).
  \end{equation}
  More generally, if $\Phi_p(q,r):=\frac{q^p-r^p}{q-r}$ is the homogenized form with $gcd(q,r)=1$, is prime and $\epsilon=\left( \frac{a^2-1}{\Phi_p(q,r)} \right)$, then
  \begin{equation}
  T_{q^p-r^p}(a)=T_{q-r}(a),\;\; U_{q^p-r^p-\epsilon (q-r)-1}(a)=0 \;\;\; mod \;\; \Phi_p(q,r).
  \end{equation}

\end{theorem}
\begin{remark}
This works also for $-q$, for example for $q=-2,a=2,\epsilon=\left( \frac{p}{3} \right)$, it says $N_p=(2^p+1)/3$ prime implies $T_{2^p}(2)=T_{3-\left( \frac{p}{3} \right)}(2)$ mod $N_p$, which in term of (1.1) is equivalent to $N_p$ divides
 $s_p-104-90\left( \frac{p}{3} \right)$. This is a weakened form and  we can derive a stronger version later, namely $N_p$ divides  $s_{p-1}-5-9\left( \frac{p}{3} \right)$. Note also we only need to code one program which will work for both  $\pm q$ provided $T_{-n}(x)$ and $U_{-n}(x)$ are implemented as $T_n(x),-U_{n-2}(x)$ as was the case with PARI-GP which we used.
\end{remark}
\begin{remark} One can compute $T_{q^p}(a)$ efficiently as $T_q^p(a)$ but $U_n(x)$ does not satisfies the compositional identity and in general they don't commute, $U_n(U_m(x)) \neq U_m(U_n(x))$. For large $n$ of no special form, $U_n$ and also $T_n$ can be computed in $O(\log n)$ steps by the usual method of writing a linear recurrence as a matrix power and applied the binary exponentiation as was observed in \cite{Lu}.  We give the formula to compute $T_{n+1}(a), U_n(a)$ mod $Q$  together via a coupled recurrence, which follows from $\omega^{n+1}=\omega^n \omega$ and (2.2) below,
\begin{equation}
\begin{pmatrix} T_{n+1}(a) \cr U_{n}(a) \end{pmatrix}=
     \begin{pmatrix} a & a^2-1 \cr 1 & a \end{pmatrix}^{n} \begin{pmatrix} a \cr 1 \end{pmatrix}\;\; mod \;\;\; Q.
\end{equation}
If $n=q^p \pm \delta$ for small $\delta$, one should use $q$-nary expansion of $n$.
\end{remark}

\begin{remark}
One can express (1.2)  in a simple Lucas-Lehmer form similar to (1.1), let
$$s_0=a,\;\; s_{n+1}=T_q(s_n),$$ if $\Phi_p(q)$ is prime, it divides $s_p-2T_{q+\left( \frac{a^2-1}{\Phi_p(q)} \right)-1}(a)$.
This is weaker than what is provable but have a simple uniform form. Sufficiency actually failed in this weak form for some small $p$ for some choices of $a$. eg.
$$q=11,p=3,a=2,M_{11}=(11^3-1)/10=133=7.19$$
$$q=-5,p=3,a=3,N=(5^3+1)/6=21=3.7$$
but this can be ruled out if we choose other starting point eg. use $a=q$. They also failed the stronger Chebyshev test.
\end{remark}
\begin{remark} Theorem (1.1) can be "seen" visually if we compute a list of values of $T_{q^p}(a)$ mod $\Phi_p(q)$ for primes $p$ up to say $200$. There is clearly a dip in the number of digits of the residues when $\Phi_p(q)$ is prime, and this is how we first saw them.
\end{remark}

Theorem 1 follows immediately from the following lemma.

 \begin{lemma} Let $Q$ be an odd prime and $a \neq 0, \pm 1$, $\omega=a+\sqrt{a^2-1},$ as before, and let $\epsilon=\left( \frac{a^2-1}{Q} \right),\delta=\left( \frac{2(a+1)}{Q} \right)$, then
    \begin{equation}
    \omega^\frac{Q-\epsilon}{2}=\delta\;\; mod \;\;Q,\;\;\;\;
    \end{equation}
    or equivalently,
    \begin{equation}
    T_{\frac{Q-\epsilon}{2}}(a)=\delta,\;\;\;U_{\frac{Q-\epsilon}{2}-1}(a)=0 \;\; mod \;\; Q,\;\;\;\;
    \end{equation}
    and this implies
   $$  T_{\frac{Q-\epsilon}{2}}(a)=\delta \;\;mod \;\; Q^2.$$
    \end{lemma}

 \begin{proof}
  We have computing mod $Q$, $$(a-1+\sqrt{a^2-1})^Q=a-1+\epsilon \sqrt{a^2-1}.$$ Multiplying by $(a-1-\epsilon \sqrt{a^2-1})$ gives
$$(2-2a)^{(1-\epsilon)/2}=(a-1+\sqrt{a^2-1})^{(Q-\epsilon)},$$
and using $(a-1+\sqrt{a^2-1})^2=2(a-1)\omega$ gives us, (note $(-1)^{\frac{1-\epsilon}{2}}=\epsilon$)
$$\omega^{(\frac{Q-\epsilon}{2})}= \left( \frac{2(a+1)}{Q} \right)\;\;mod\;\; Q.$$
But we have (without mod $Q$ )
$$\omega^{\frac{Q-\epsilon}{2}}=T_{\frac{Q-\epsilon}{2}}(a) +U_{\frac{Q-\epsilon}{2}-1}(a) \sqrt{a^2-1},$$
by  (2.2) below,
which give the equivalent (1.5).
 \end{proof}
 \begin{remark}
 Writing $n=\frac{Q-\epsilon}{2}$ , since $\omega$ is a unit , so is $\omega^n$, we must have the Pell's equation $\omega^n \overline{\omega}^n=T_n(a)^2-(a^2-1)U_{n-1}(a)^2=1$. So we have $Q^2$ divides
    $T_n(a)^2-1=(T_n(a)-\delta)(T_n(a)+\delta)$. Since $Q$ divides $T_n(a)-\delta$, its prime divisor  cannot divide $T_n(a)+\delta$, so we always have
    $T_n(a)=\delta$ mod $Q^2$.
 \end{remark}

\begin{proof}(Proof of theorem) Specialize to $Q=\Phi_p(q)=\frac{q^p-1}{q-1}$ ($q$ may be negative) in (1.4) gives
$$ \omega^\frac{q^p-1-\epsilon(q-1)}{2(q-1)}=\delta,$$
and raising to the $2(q-1)$ power (this lose information) gives $\omega^{q^p-1-\epsilon(q-1)}=1$, which is the same as
$$T_{q^p}(a)=T_{\epsilon(q-1)+1}(a)=T_{q+\epsilon-1}(a),\;\;\; U_{q^p-\epsilon(q-1)-2}(a)=0\;\; mod \;\; Q.$$
Proof of (1.3) is similar. Note the T-test is independent of $\epsilon$.
\end{proof}

\begin{remark} Using (1.4), one can find similar divisibility criteria  of the same sequence for many class of primes, eg
$q=\pm 2,a=2$  and $s_n$ the usual Lucas-Lehmer sequence (1.1),  we have

$M_p=2^p-1$ prime implies $M_p$ divides $s_{p-2}$.

$N_p=(2^p+1)/3$ prime implies $N_p$ divides $s_{p-1}-5-9 \left( \frac{p}{3} \right)$

$n>2,M_n=3.2^n-1,$   prime implies  $M_n$ divides $(s_{n-1}^3-3s_{n-1}-4)$

$n>2,N_n=3.2^n+1,$  prime implies    $N_n$ divides $(s_{n-1}+1)(s_{n-1}-2)$, etc...

The last two follows from setting $a=2$, $T_{3.2^{n-1}}(2)=T_1(2)$ and $T_{3.2^{n-1}}(2)=1$.
\end{remark}

\begin{remark} For a cubic example, let $q= \pm 3, a=2$, and $s_0=2,s_{n+1}=s_n(4s_n^2-3),$
then $M_p=(3^p-1)/2$ prime implies it divides $s_p-26$ and $N_p=(3^p+1)/4$ prime implies it divides
$s_p-194+(-1)^{(p-1)/2}168$.
\end{remark}

\section{Chebyshev primality test} If $Q$ is a prime and $a$ an integer with $gcd(a^2-1,Q)=1$, by (1.5), we have
\begin{equation}T_{\frac{Q-\epsilon}{2}}(a)=\left( \frac{2(a+1)}{Q} \right),\;\; U_{\frac{Q-\epsilon}{2}-1}(a)=0\;\;\; mod \;\;Q.
\end{equation}

Clearly all odd primes $Q$  pass this test to every base $a$. We shall called an odd non-prime integer $Q$, with $gcd(Q,a^2-1)=1$, which pass this test a Chebyshev pseudoprime to the base $a$. It depends only on $a$ mod $Q$ but there is no subgroup structure. Chebyshev pseudoprimes are always squarefree except for some prime squared.They are rare and seems rarer than Fermat pseudoprimes. There are only seven of them to the base $2$ upto $20000$,
$$23.43,\;\;\;37.73,\;\;\;103^2,\;\;\;61.181,\;\;\;5.7.443,\;\;\;97.193,\;\;\;31.607.$$
 Is there a Chebyshev pseudoprime to every base mod $Q$ ? A Sierpi\'nski number \cite{WS} is a positive odd integer $k$ such that $N_n=k2^n+1$ is composite for every $n \ge 1$. $k_0=78557$ is the smallest known Sierpi\'nski number, because every $N_n=k_02^n+1$ is divisible by one of $\{3, 5, 7, 13, 19, 37, 73\}$. It may be possible that $N_n$ fail a Chebyshev test for every $n$ for some $a$. Since $N_n=1$ mod $8$ for $n \ge 3$. We get $\epsilon=\delta=1$ if we pick $a=3$. So if $s_0=3, s_{k+1}=2s_k^2-1$, and $N_n^2=(k2^n+1)^2$ does not divide $T_k(s_{n-1})-1$ for every $n \ge 3$, then $k$ is Sierpi\'nski. Note $N_{n+1}=2N_n-1$.
It is open if any of the following five numbers $21181, 22699, 24737, 55459, 67607$ is Sierpi\'nski.

A Chebyshev pseudoprime for the base $a$ is  also a weak Chebyshev pseudoprime  as defined in \cite{MTW} ie. $T_Q(a)=a$ mod $Q$ since the condition on $U$ means $\omega^{Q-\epsilon}=1$ or $\omega^Q=\omega^\epsilon$ and taking trace gives $T_Q(a)=T_1(a)=a$ mod $Q$. There are composites which pass the weak test for all base $a$ OEIS A175530, but all of them fail the strong Chebyshev test for all base from $2$ to $10$.

A square-free $Q$ which pass the $T$ test will also pass the $U$ test.[Proof: We have $T_{Q-\epsilon}(a)=1$ so that
$(\omega^{(Q-\epsilon)/2}-\overline{\omega}^{(Q-\epsilon)/2})^2=0$ and squarefreeness of $Q$ implies $U_{(Q-\epsilon)/2-1}(a)=0$.]
There are many non square-free integers which pass the $T$ test but the only non squarefree integer which can pass both tests are square of prime (Proof ?). So  the second part only serve to rule out non squarefree integer and this is relevant since there is no known efficient  algorithm to detect squarefreeness. However we can always rule out perfect square as input

If $(Q-\epsilon)/2=2^tQ_1$ is even , we can look at the profile $[T_{Q_1}(a), T_{2Q_1}(a),...,T_{(Q-\epsilon)/2}(a)]$ as in the strong pseudoprime test. Since $T_2(x)=2x^2-1$,if  there is a $1$ not preceded by $\pm 1$ or a $-1$ not preceded by $0$,
$Q$ cannot be prime. For the seven pseudoprimes above, the profiles are
$$[1],[0,-1],[9083,0,-1,1],[0,-1,1,1,1],[8416,4431,8861,1],$$
$$[14063,17370,18527,387,1],[18791,1301,18720,0,-1,1],$$
so the strong test rule out $5.7.443$ and $97.193$ as primes. For square free $Q$, $-1$ is always preceded by $0$, since $(\omega^m+\overline{\omega}^m)/2=-1$ implies $(\omega^{m/2}+\overline{\omega}^{m/2})^2=0$ mod $Q$.

We note that for $a \neq 0 \pm 1$, $(\omega_a,\overline{\omega}_a)$ forms a Lucas pair in the sense of \cite{BHV}, since
$\frac{\omega_a}{\overline{\omega}_a}$ is not a root of unity. The associated Lucas number $u_n(\omega_a,\overline{\omega}_a)=U_{n-1}(a)$. It seems to follow from \cite{BHV} that for every $n>1$, $U_n(a)$ has a primitive divisor, ie. there is a prime $p$ which divides $U_n(a)$ but not $a(a^2-1)U_0(a)...U_{n-1}(a)$.

\subsection{Multiplicative order  and sufficiency test} Many of the necessity criteria seems to be sufficient in the range we can compute.
It could be that when $Q$ is composite, the residue behave randomly and the chance they give divisibility is $1/Q$ which is very small so we never see them.

Let $\omega=a+\sqrt{a^2-1}$ be the canonical unit. For an integer power $n$, we must have $\omega^n=P(a)+Q(a)\sqrt{a^2-1}$ for some  $P(x),Q(x) \in \mathbb{Z}[x]$ but for $\omega$, they are just Chebyshev polynomials \cite{WC},or just by induction,
\begin{equation}
\omega^n=T_n(a)+U_{n-1}(a)\sqrt{a^2-1},
\end{equation}
 and $n$ may be negative. Writing $\omega^{n+1}=\omega^n\omega$ gives the recurrence formula in (1.3).

For an odd integer $Q$, let $O_Q(\omega)$ be the multiplicative order of $\omega_a$ mod $Q$, ie. the smallest integer $m$ such that $\omega^m=1$ mod $Q$. This is thus the same as the smallest integer $m$ such that $T_m(a)=1$ and $U_{m-1}(a)=0$ mod $Q$. Note that for a prime $Q$ or a Chebyshev pseudoprime $Q$, we have $\omega_a^{Q-\epsilon}=\delta^2=1$ so that $O_Q(\omega_a)$ divides one of $Q \pm 1$, in particular it divides $Q^2-1$ and $O_Q(\omega) \le Q+1$.

There seems to be only one argument to prove primality of $Q$. One shows that $\omega$ has multiplicative order $Q \pm 1$ and hence $Q$ cannot have any non trivial divisor, since it will have the same order in $\mathbb{F}_p[\sqrt{a^2-1}]$ for the smallest prime $p$ dividing $Q$ of size $t^2<Q \pm 1$. We can determine the order if we know the complete factorization of $Q \pm 1$.
\begin{lemma} Let $Q$ be an odd integer and $a \neq 0, \pm 1$, and assume $\left( \frac{2(a+1)}{Q} \right)=-1$.
Let $\epsilon=\left( \frac{a^2-1}{Q} \right)$. If $Q$ is prime, then $T_{(Q-\epsilon)/2}(a)=-1$ and $U_{(Q-\epsilon)/2-1}(a)=0$.
Conversely if we know the complete factorization $Q-\epsilon=\prod_{j=0}^k q_j^{n_j},j=0,..,k$ where $q_0=2$, and we have
$T_{(Q-\epsilon)/2}(a)=-1$ and $U_{(Q-\epsilon)/2-1}(a)=0$, and also $T_{(Q-\epsilon)/q_j}(a)\neq 1$ or $U_{(Q-\epsilon)/2-1}(a) \neq 0,$
for $j=1,...,k$, then $Q$ is prime.
\end{lemma}

For $Q=2^p-1$, and $a=2$, we get $\omega^{2^{p-1}}=-1$ so that $O_Q(\omega_2)=2^p=Q+1$. So $2^p-1$ is prime if and only if $O_2(\omega_2)=Q+1$ if an only if $T_{2^{p-1}}(2)=-1$ and this is equivalent to $T_{2^{p-2}}(2)=0$ mod $Q$.
Instead of starting with $a=2$, we can choose any $a$ of the form $a=1+x^2$ so that $a+1=3y^2$ or $a+1=6y^2$ , we then have $\epsilon=-1=\delta$
we still have $2^p-1$ is prime if and only if $T_{2^{p-2}}(a)=0$ mod $Q$. This condition turns out to be necessary and sufficient and is given in OEIS A18844.

Lemma 2.1 is just the analogue of the usual computational definition of the existence of a primitive root in the case of $\mathbb{Z}/Q^*$ but there is  one basic difference here, since changing base $a$ means changing the group $\mathbb{Z}[\sqrt{a^2-1}]^*$ also.
We can change $a$ until we get the correct order.

\begin{example} Let  $a=2$ and $r=r_1,...r_k$ be an odd square free integer not divisible by $3$  and let $\delta=2-r \in \{0,1\}$ mod $3$, $N=2n+\delta$ and $Q=r2^{N}-1$. Let $s_0=2, s_{n+1}=s_n^2-2$. Then $Q$ is prime implies $Q$ divides
$T_r(s_{N-2}/2)$. Conversely if $Q$ is odd integer of the form $r2^N-1$, and divides $T_r(s_{N-2}/2))$  and in addition,  $T_{(r/r_j)}(s_{N}/2) \neq 1$ mod $Q$, for $j=1,...k$, then $Q$ is prime.
\end{example}
\begin{proof} The value of $\delta$ were chosen such that $Q=1$ mod $3$, so for $a=2$, $\epsilon=-1$ and $\left( \frac{2(a+1)}{Q} \right)=-1$ and we have $\omega^{(Q-\epsilon)/2}=\omega^{r2^{N-1}}=-1$, so that
$T_{r2^{N-2}}(2)=0$. It also implies the order $O_Q(\omega)=t_1...t_k2^N$ where $t_j$ divides $r_j$.
If $T_{(r/r_j)}(s_{N}/2) \neq 1$, $\omega^{(r/r_j)2^N} \neq 1$,
we must have $t_j=r_j$ and $O_M(\omega)=M+1$.
\end{proof}
If we let $r=5$,then for $n$ up to $3000$, there are $29$ primes and $5$ of them at $n=2,18,32,1638,2622$, fail the sufficiency tests.

We also have an order $q$ version

\begin{example} Let $Q=12q^n+1$ be prime where $q$ is an odd prime, then $T_{3q^n}(2)=0$ mod $Q$. Conversely if an odd integer $Q$ is of the form $12q^n+1$  satisfies $T_{3q^n}(2)=0$, and in addition $T_{4q^n}(2), T_{12q^{n-1}}(2)$ are all not $1$ mod $Q$, then $Q$ is prime.
\end{example}
\begin{proof} We have $Q=1$ mod $3$ and $5$ mod $8$. So if $a=2$, $\epsilon=1, \delta=-1$ so that we have $\omega^{6q^n}=-1$.
So $T_{3q^n}(2)=0$ and $O_Q(\omega)=4.3^{t_1}q^{t_2}, t_1 \le 1, t_2 \le n$. It is $12q^n=Q-1$ iff $T_{12q^{n-1}}(2) \neq 1$ and $T_{4q^n}(2) \neq 1$.
\end{proof}
For $q=5$, $Q$ is prime when
$$n=1[1,1],5,7,18,19,23,46,51,55,69[1,*],126[*,1],469,1835[*,1],3079,3249,4599,4789$$ but the primality proof failed
for $1,69,126,1835$, but we get a proof when we change base.
\smallskip

Recall that a Proth's number $N=k2^n+1$, where $k$ is odd and $k<2^n$, is prime if and only if there ia an integer $a$ such that $a^{(N-1)/2}=1$ mod $N$.

We have an  exact analogue
\begin{lemma} Let $N=k2^n+1$ where $k$ is odd and $k<2^n$. Let $a$ be such that $\epsilon=-\delta=1$, then $N$ is prime if and only if
it pass the  Chebyshev test, ie.  $\omega_a^{(N-1)/2}=-1$ or equivalently $T_{k2^{n-1}}(a)=-1, U_{k2^{n-1}-1}(a)=0$ mod $N$.
\end{lemma}
\begin{proof} The necessity is just Chebyshev test. Conversely $\omega^{(N-\epsilon)/2}=\omega^{k2^{n-1}}=-1$ mod $N$ implies the same mod any prime $p$ dividing $N$, which implies $p+1 \ge O_p(\omega) \ge 2^n$, which means every prime divisor of $N$ is  greater than $\sqrt{N}$.
\end{proof}
A special case of this is a question in MathOverflow \cite{Te1}, where we  set $a=4$ (see also answer by Ian Algol). The requirement $\epsilon=-\delta=1$ translate to
$\left( \frac{5}{N} \right)=-1$ and $\left( \frac{N}{3} \right)(-1)^{(N-1)/2}=-1$ since $N=1$ mod $8$ for $n>2$, and note that $P_n(x)=2T_n(x/2)$.

For any fixed $k$ and $n$, there is always some choice of $a$ to give a necessary and sufficient condition. What we want is for a fixed $k$ to find an $a$ which works for all $n$ but for  $k=3$, this does not seem to be possible.

\begin{example} In the same way if $n>1$ and $F_n=2^{2^n}+1$ and  set $a=4$, we have $\epsilon=1=-\delta$, so $\omega_4^{2^{2^{n-1}}}=-1$
so that $O_{F_n}(\omega_4)=2^{2^{n}}=F_n-1$ and also $\omega_4^{2^{2^n-2}}+\overline{\omega}_4^{2^{2^n-2}}=0$. So $F_n$ is prime if and only if $T_{2^{2^n-2}}(4)=0$ mod $F_n$. In Lucas-Lehmer term if $s_0=8, s_{n+1}=s_n^2-2$, then $F_n$ is prime if and only if $F_n$ divides $s_{2^n-2}$.
\end{example}

\section{Chebyshev polynomials are just even and odd part of binomials}
\subsection{Twisting Chebyshev polynomials}
The Chebyshev test depends on the unit $a+\sqrt{a^2-1}$. It is natural to wonder if we get something new  using the negative units $a+\sqrt{a^2+1}$ instead.
 If we let $\omega_x=x+\sqrt{x^2+1}$, we will have $\omega_x\bar{\omega}_x=-1$.
If we define the polynomial $\omega_x^n=S_n(x)+V_{n-1}(x)\sqrt{x^2+1}$, we have $S_1(x)=x, V_0(x)=1$ and
$$S_{n+1}(x)=xS_n(x)+(x^2+1)V_{n-1}(x)$$
$$V_n(x)=S_n(x)+xV_{n-1}(x),$$
or
\begin{equation}
\begin{pmatrix} S_{n+1}(a) \cr V_{n}(a) \end{pmatrix}=
     \begin{pmatrix} a & a^2+1 \cr 1 & a \end{pmatrix}^{n} \begin{pmatrix} a \cr 1 \end{pmatrix}.
\end{equation}
Again $S_n(x)=\frac{\omega_x^n+\bar{\omega}_x^n}{2}$. $S_n(x)$ is the same as $T_n(x)$ with all coefficients positive and $V_n$ is $U_n$ with all sign positive. So they are congruent mod $2$. Roots of $S_n,V_n$ are $i$ times those of $T_n,U_n$,
which follows from $S_n(x)=i^nT_n(x/i),  V_n(x)=i^nU_n(x/i)$. We also have the finite golden ratio
$$ x+\cfrac{1}{x+\cfrac{1}{x+...\cfrac{...}{x+\cfrac{1}{x}}}}=\frac{V_n(x/2)}{V_{n-1}(x/2)}=g_{A_n,1}(x),$$
which is the hyperbolic version of
$$ x-\cfrac{1}{x-\cfrac{1}{x-...\cfrac{...}{x-\cfrac{1}{x}}}}=\frac{U_n(x/2)}{U_{n-1}(x/2)}=f_{A_n,1}(x),$$
which converge to $\frac{x+\sqrt{x^2 \mp 4}}{2}$ which are positive branch to the inverse of the simplest quadratic $f(x)=x \mp \frac{1}{x}.$  Here $f_{A_n,1}(x)=\frac{\det(xI-A_n)}{\det(xI-A_{n-1})},g_{A_n,1}(x)=\frac{Per(xI-A_n)}{Per(xI-A_{n-1})}$, are Cauchy interlacing ratio for the path graph on $n$ vertices to that with one end point deleted.

Also we have the Pell's equation
$$S_n^2(x)-V_{n-1}^2(x)(x^2+1)=(-1)^n,$$
so that
$\omega_x^n=S_n(x)+\sqrt{S_n(x)^2-(-1)^n}$ so that $\omega_x^{n}=\omega_{S_{n}(x)}$, for odd $n$,
and $S_m(S_n(x))=S_{mn}(x)$ for $n$ odd. Note $S_3(S_2(x))=32x^6+48x^4+30x^2+7 \not= S_6(x)=32x^6+48x^4+18x^2+1=S_2(S_3(x))$. We still have compositional commutativity $S_n(S_m)=S_m(S_n)$ for $m,n$ both odd. Is there a  primality test based on iterating this ?

 \subsection{Square root of even and odd part of Chebyshev polynomials:  s-Eulerian and Erhart polynomials} The polynomial
 $$P_n(x)=1 + \binom{n}{2} x + \binom{n}{4} x^2 + \binom{n}{6} x^3 + \binom{n}{8} x^4 +\ldots + \binom{n}{2\lfloor\tfrac{n}{2}\rfloor} x^{\lfloor \frac{n}{2}\rfloor},$$
 is an $s$-Eulerian polynomial and is known to be real rooted.
 Since $degree(P_n)=\lfloor n/2 \rfloor$, we expect $P_n$ to interlace $P_{n-2}$ which was supported by computations.
 A problem posted in MO (Luis Ferroni) asked to prove that the polynomial $Q_n(x):=P_n(x)-nx$ which is the Ehrhart $h$*-polynomial of the hypersimplex $\Delta_{2,n}$ is also real rooted.

This was solved explicitly by  Fedor Petrov who observe that
 $$2P_n(-x)=(1+i\sqrt{x})^n+(1-i\sqrt{x})^n.$$
 If we expand $2P_1(x)*2P_{n-1}(x)$, we get the recursion
   $P_n(x)=2P_{n-1}(x)+(x-1)P_{n-2}(x)$ so that we have a continued fraction
 $$\frac{P_n(x)}{P_{n-1}(x)}=2+\frac{x-1}{\frac{P_{n-1}(x)}{P_{n-2}(x)}}=2+\cfrac{x-1}{2+\cfrac{x-1}{...+\cfrac{...}{2+\cfrac{x-1}{x+1}}}}$$
 which gives a fast way to compute $P_n(x)$.

We also have by expanding $2P_2(-x).2P_{n-2}(-x)$, $$ P_n(x)=2(1+x)P_{n-2}(x)-(1-x)^2P_{n-4}(x)$$ which gives

$$\frac{P_n(x)}{P_{n-2}(x)}=2(1+x)-\frac{(1-x)^2}{\frac{P_{n-2}(x)}{P_{n-4}(x)}}=2(1+x)-\cfrac{(1-x)^2}{2(1+x)-\cfrac{(1-x)^2}{...-\cfrac{...}{2(1+x)-\cfrac{(1-x)^2}{x+1}}}}$$
which allow us to prove interlacing inductively by looking at the graphs.

It is obvious from the MO solution that
$P_n$ is essentially a squareroot of Chebyshev. $T_n(x)=x^nP_n((x^2-1)/x^2)$ so we know their roots explicitly from those of $T_n$, and they are given by
$$-\tan((k+1/2)\pi/n)^2,k=0,...\lfloor (n/2 \rfloor-1.$$

\subsubsection{Deriving Chebyshev polynomilas  from the binomial $(1+x)^n$}
We let
$$(1+i\sqrt{x})^n=P_n(-x)+iQ_n(-x)\sqrt{x},$$
so that
$$Q_n(x)=\sum_{j=0}^{\lfloor n/2-1 \rfloor} { n \choose 2j+1}x^j, \; \; \;
P_n(x)=\sum_{j=0}^{\lfloor n/2 \rfloor} { n \choose 2j}x^j,$$
$$2P_n(-x)=(1+i\sqrt{x})^n+(1-i\sqrt{x})^n$$
$$2iQ_n(-x)\sqrt{x}=(1+i\sqrt{x})^n-(1-i\sqrt{x})^n,$$

The tangent substitution $x=(1-y^2)/y^2$ gives
$$T_n(y)=y^nP_n(\frac{y^2-1}{y^2}),U_{n-1}(y)=y^{n-1}Q_n(\frac{y^2-1}{y^2}),$$
since $(x+\sqrt{x^2-1})^n=T_n(x)+U_{n-1}(x)\sqrt{x^2-1}$, and we also have the dual

$(x+\sqrt{x^2+1})^n=S_n(x)+V_{n-1}(x)\sqrt{x^2+1}$.

The hyperbolic tangent substitution $x=(1+y^2)/y^2$,  gives
$$ S_n(y)=y^nP_n(\frac{y^2+1}{y^2}),  V_n(y)=y^{n-1}Q_n(\frac{y^2+1}{y^2}).$$

Multiplying $(1+i\sqrt{x})^n$ with $(1-i\sqrt{x})^n$  gives $(1+x)^n=P_n(-x)^2+xQ_n(-x)^2$.
This does not look obvious and it implies a series of binomial identities.
Obviously we also have
$$(1+x)^n=\sum_{j=0}^{\lfloor(n/2) \rfloor}{n \choose 2j}x^{2j}+x\sum_{j=0}^{\lfloor((n-1)/2) \rfloor}{n \choose 2j+1}x^{2j}=P_n(x^2)+xQ_n(x^2).$$
which give  the functional equation
$$P_n(x^2)+xQ_n(x^2)=P_n(-x)^2+xQ_n(-x)^2 $$
but note $P_n(x^2) \not= P_n(-x)^2$. Are there other such pairs of polynomials ?

So the Chebyshev polynomials $T_n,U_n$ are essentially just the even and odd part of the binomial polynomial $(1+x)^n$.

\subsection{Aside : Extending Mersenne primes search to cyclotomic primes} The Mersenne prime $M_p=2^p-1$ can be expressed as $M_p=\Phi_p(2)$ where $\Phi_m(x)$ denotes the $m$th cyclotomic polynomial. It  is natural and seemingly useful to generalize the Mersenne prime  to more general cyclotomic primes. Let $\Phi_m(x,y)=\Phi_m(x/y)y^{\phi(m)}$ be the homogenized cyclotomic polynomial.
Since $\Phi_m(x)$ are irreducible with fixed divisor $1$, ($\phi_m(1)=1$, except $\phi_p(1)=\phi_{p^2}(1)=p$) , there should be infinitely many primes of the form $\Phi_m(r,s)$ as $r,s$ varies, by Schinzel's conjecture. One expects some form of uniform distribution so that we still get infinitely many primes if we fixed $r,s$ with $|s| < r$ and $gcd(r,s)=1$ and varies $m$. For example, one can widen the difficult question for the infinitute of Mersenne primes to those of the form $\Phi_m(2,1)$ (note not the same as $2^m-1$), and more generally, are there always infinitely many primes of the form $\Phi_m(r,s)$ for foxed $|s|<r,gcd(r,s)=1$?

The Mersenne prime search also generalize naturally to multi prime-tuple search for primes of the form $\Phi_{p_1..p_k}(r,s)$ over distinct odd primes $p_1,...,p_k$. Non-squarefee $m$ with $\Phi_m(r,s)$ prime are rare, and only occurs when $m=p^2$ and $\Phi_{2m}(x)=\Phi_m(-x)$ for $m$ odd.  We have  found a $2152$ digit prime $\Phi_{2021}(4,13)$. Mersenne numbers are interesting partly because they have a simple form $\Phi_p(r,s)=\frac{r^p-s^p}{r-s}$ and this generalizes to  $\Phi_m(r,s)=\prod_{d|m} (r^d-s^d)^{\mu(n/d)}$ which follows from $x^m-1=\prod_{d|m} \Phi_d(x)$ by inclusion/exclusion,
so we can express the above prime, in nicer form
$$\Phi_{2021}(4,13)=\frac{(4-13)(4^{2021}-13^{2021})}{(4^{43}-13^{43})(4^{47}-13^{47})}.$$
We also found a  three-tuple  $5599$ digit example
$$\Phi_{13.17.29}(11,-4)=\frac{(11^{6409}+4^{6409})(11^{13}+4^{13})(11^{17}+4^{17})(11^{29}+4^{29})}{(11+4)(11^{221}+4^{221})(11^{377}+4^{377})(11^{493}+4^{493})}.$$
For $\pi$ day we found 
$$\Phi_{79.89}(3,-14)=\frac{(3+14)(3^{79.89}+14^{79.89})}{(3^{79}+14^{79})(3^{89}+14^{89})}$$
which is a $7868$ digit prime.

This seems to be very useful for outreach purpose to impress the (even educated) public, perhaps more impressive than  Mersenne primes. It is not at all obvious that the RHS is even an integer and it may seem  mysterious that the it will somehow cancel out and left with a single term  which is a prime. However we do not have a simple sufficiency test like Lucas-Lehmer for Mersenne prime.

The product formula for example for $k=2$,
$\Phi_{p.q}(r,s)=\frac{(r^{p.q}-s^{p.q})-(r-s)}{(r^p-s^p)(r^q-s^q)}$ means we are searching   along prime exponents which  does not seem to be governed by the usual conjectures.

It does not seem easy to even prove that there is at least one such prime for fixed $r,s$ with $gcd(r,s)=1$. Maybe the only way is to prove positive density but they seem very sparse.

\section*{Acknowledgements} This works started  when we read some posting on MathOverflow of user Pedja Terzi\'c and realized that the function he defined is essentially Chebyshev polynomial $T_n(x)$ and that the compositional identity $T_n(T_m(x))=T_{nm}(x)$ means there is an implied  $q$-nary Lucas-Lehmer iteration algorithm. Numerical experimentation then lead us to the statement of theorem 1.1. The $q$-nary Lucas Lehmer is essentially known in many posting by Pedja Terzi\'c \cite{Te1, Te} and these can all be derived from our main  Lemma 1.6.


\bibliographystyle{amsplain}

\begin{thebibliography}{4}
\bibitem{Lu} Lucia, MathOverflow, Is there an explicit formula for Chebyshev polynomials mod $x^r-1$.
\end{thebibliography}

\end{document}